%% file: main_ecp.tex
\documentclass[ECP]{ejpecp}

\usepackage{notation}

\usepackage{todonotes}

\SHORTTITLE{Sharper Perturbed-KL Exponential Tail Bounds for Beta and Dirichlet Distributions}

\TITLE{Sharper Perturbed-Kullback-Leibler Exponential Tail Bounds for Beta and Dirichlet Distributions
    } 

\AUTHORS{%
Pierre Perrault\footnote{IDEMIA.  \EMAIL{pierre.perrault@outlook.com}}
}


\AMSSUBJ{60E15; 62E17} 

\SUBMITTED{...} 
\ACCEPTED{...} 

\VOLUME{0}
\YEAR{2025}
\PAPERNUM{0}
\DOI{10.1214/YY-TN}

\ABSTRACT{This paper presents an improved exponential tail bound for Beta distributions, refining a result in \cite{henzi2023some}. This improvement is achieved by interpreting their bound as a regular Kullback–Leibler (KL) divergence one, while introducing a specific \emph{perturbation} $\eta$ that shifts the mean of the Beta distribution closer to zero within the KL bound. Our contribution is to show that a larger perturbation can be chosen, thereby tightening the bound. We then extend this result from the Beta distribution to Dirichlet distributions and Dirichlet processes (DPs). 
}

\begin{document}
\input{sections/ecp_version}

\bibliographystyle{amsplain}
\bibliography{ref}

\end{document}

%% file: sections/ecp_version.tex
\section{Introduction}
\label{sec:introduction}

The Beta distribution, denoted $\text{Beta}(a, b)$ for parameters $a, b > 0$, is a fundamental continuous probability distribution supported on the interval $[0, 1]$. It arises naturally as a prior in Bayesian statistics \cite{castillo2017polya}, and in contexts involving proportions and random processes constrained to finite intervals \cite{frankl1990some}.

The upper tail probability of a Beta-distributed random variable $X \sim \text{Beta}(a, b)$ is defined as

$$
\PP{X \geq u} \triangleq \frac{\Gamma(a + b)}{\Gamma(a) \Gamma(b)} \int_u^1 x^{a - 1}(1 - x)^{b - 1} \, dx, \quad \text{for } 0 \leq u \leq 1,
$$
\noindent
where $\Gamma(z) \triangleq \int_0^\infty t^{z - 1} e^{-t} \, dt$ denotes the gamma function. Despite its simple definition, computing \(\PPs{X\sim \text{Beta}(a, b)}{X \geq u}\) is generally intractable~\cite{dutka1981incomplete}. Nonetheless, accurately estimating this quantity finds applications in areas such as large deviation theory~\cite{zhang2020non} and adaptive Bayesian inference~\cite{elder2016bayesian}. As a result, a variety of closed-form exponential upper bounds have been developed. We summarize some of the most notable ones in Table~\ref{tab:bounds}.

Among these bounds, Kullback–Leibler (KL)-type bounds stand out for their asymptotic sharpness. Specifically, for \(x \in (0, 1)\), the tail probability admits the following large deviation approximation (see Table~\ref{tab:bounds} for the definition of the $\text{kl}$ function):
\[
\log \PPs{X\sim \text{Beta}(nx, n(1-x))}{X \geq u} \sim_{n \to \infty} -n \kl{x}{u}, \quad \text{for } u \in (0, 1),\quad\text{(see, e.g.,~\cite{lynch1987large}).}
\]

We refer to a particularly effective refinement of the standard KL upper bound, recently proposed in~\cite{henzi2023some}, as the \emph{perturbed KL bound}. This bound improves upon both Hoeffding-type~\cite[Cor.~5]{henzi2023some} and Bernstein-type~\cite[Lem.~6]{henzi2023some} inequalities. The idea is to introduce a perturbation parameter \( \eta > 0 \) into the mean within the kl, shifting it from the standard Beta mean \( a/(a + b) \) to a smaller value \( (a - \eta)/(a + b - \eta) \). This modification yields a tighter divergence, effectively sharpening the bound by moving the mean slightly away from the threshold \( u \).

The first result of this work\footnote{This first result was also presented in \cite{perrault2025perturbed}.} establishes that the perturbed KL bound for Beta distributions from \cite{henzi2023some} can be improved by allowing the perturbation parameter to exceed the threshold \( 1 + \frac{a-1}{b+1} \).

\begin{table}[t]
\centering
\begin{tabular}{|c|c|c|}
\hline
Bound on $\log\PPs{X\sim \text{Beta}(a, b)}{X \geq u}$ & Validity range & Type
\\
  \hline
  \transparent{0}
\fbox{\transparent{1}${-2(a+b+1)\pa{u - \frac{a}{a+b}}^2}$} & \multirow{5}{*}{$\frac{a}{a+b}\leq u\leq 1$} & Hoeffding \cite{marchal2017sub} 
  \\
  \cline{1-1}\cline{3-3}
  \transparent{0}\fbox{\transparent{1}
  \(-\frac{\pa{(a+b)u - {a}}^2}{\frac{2ab}{a+b+1}+\frac{4 \pa{(a+b)u - {a}}\pa{b-a}^+}{3(a+b+2)}} \)} && Bernstein \cite{skorski2023bernstein} \\
  \cline{1-1}\cline{3-3}\transparent{0}\fbox{\transparent{1}
  $-(a+b)\kl{\frac{a}{a+b}}{u}$}& & Kullback–Leibler (KL) \cite{dumbgen1998new} \\
  \hline   \transparent{0}\fbox{\transparent{1}
  \(\begin{array}{cc}
  -(a+b-\eta)\kl{\frac{a-\eta}{a+b-\eta}}{u}, \\
  \text{for }a\geq 1, b > 0, \eta= 1+\frac{a-1}{b+1}. 
    \end{array}\)} & {$\frac{a-\eta}{a+b-\eta}\leq u\leq 1$} & {perturbed KL \cite{henzi2023some}}\\
  \hline
\end{tabular}
\caption{Table of well-known exponential upper bounds for the Beta distribution, where $\kl{p}{q}\triangleq p\log(p/q)+(1-p)\log((1-p)/(1-q))\in [0,\infty]$ for $p,q\in [0,1].$}
\label{tab:bounds}
\end{table} 

On the other hand, the Dirichlet distribution and its infinite-dimensional counterpart, the Dirichlet process (DP) ---
 originally introduced by Ferguson in the 1970s~\cite{ferguson1973bayesian} --- also frequently rely on tail probability bounds in settings such as multi-armed bandits~\cite{belomestny2023sharp}, reinforcement learning~\cite{tiapkin2022dirichlet}, and Bayesian bootstrap methods~\cite{rubin1981bayesian}. A widely used tail bound in this context is a direct extension of the standard KL bound from \cite{dumbgen1998new}, transitioning from the Beta distribution to DPs (more precisely, to Dirichlet-weighted sums) by replacing the binary KL with the multivariate KL (see \cite{perrault2024new}).

This naturally leads to the question: Can the perturbed KL-type bounds --- proven effective for Beta distributions --- be generalized to the Dirichlet setting and, more broadly, to DPs?

The second result of the present work is to provide an affirmative answer to this question.  Building on our refined Beta tail bound, we then extend it to DPs. 
Before presenting our two contributions, we introduce the necessary notation.

\subsection{Notation and preliminaries}
We consider a compact metric space \( \Omega \) equipped with its Borel \( \sigma \)-algebra \( \mathcal{B}(\Omega) \). Let \( \mathcal{M}(\Omega) \) (resp. \( \mathcal{M}_1(\Omega) \)) denote the space of finite non-negative (resp. probability) measures on \( \Omega \), and let \( \mathcal{B}(\mathcal{M}_1(\Omega)) \) denote the Borel \( \sigma \)-algebra induced by the weak topology on \( \mathcal{M}_1(\Omega) \). The set of continuous functions \( f: \Omega \to \mathbb{R} \) is denoted by \( \mathcal{C}(\Omega) \). The Kullback–Leibler (KL) divergence between two probability distributions \( \mu , \nu \in \cM_1(\Omega) \) is defined as:
\[
\KL{\mu}{\nu} \triangleq
\begin{cases}
\int \log \left( \frac{d\mu}{d\nu} \right) d\mu & \text{if } \mu \ll \nu, \\
\infty & \text{otherwise}.
\end{cases}
\]
We consider a Dirichlet Process (DP) on \( \Omega \), characterized by a scale parameter \( \alpha > 0 \) and a base distribution \( \nu_0 \in \mathcal{M}_1(\Omega) \), whose measure's law is denoted as \( \text{DP}(\alpha \nu_0) \), and whose realization \( X \) is a random probability measure on the space \( \Omega \). We assume, without loss of generality, that \( \nu_0 \) is supported on \( \Omega \).
The original definition of the DP says that for any finite measurable partition $A_1\sqcup\dots\sqcup A_k=\Omega, \pa{X(A_1),\dots,X(A_k)}\sim\text{Dir}\pa{\alpha\nu_0(A_1),\dots,\alpha\nu_0(A_k)}$, where $\text{Dir}(a_1,\dots,a_k)$ denotes the Dirichlet distribution with
parameters $a_1,\dots,a_k$.
Another important representation of the DP is related to the characterization of the Gamma distribution by \cite{lukacs1955characterization}, and is expressed as \( X = G/G(\Omega) \), where \( G \sim \mathcal{G}(\alpha \nu_0) \) is the standard Gamma process with shape parameter \( \alpha \nu_0 \), i.e.,
$G(A) \sim \text{Gamma}(\alpha\nu_0(A), 1)$ for any measurable set $A \subset \Omega$ \cite{casella2005christian}. An important structural property is the stochastic independence between \( X = G / G(\Omega) \) and the total mass \( G(\Omega) \).

For a probability distribution \( \nu \in \mathcal{M}_1(\Omega) \), one of the central information-theoretic measures we are considering is defined as an infimum of KL divergences: for some real-valued continuous function \( f \in \mathcal{C}(\Omega) \) and some \( u \in \mathbb{R} \), we define
\[
\mathcal{K}_{\inf}(\nu, u, f) \triangleq \inf_{\mu \in \mathcal{M}_1(\Omega), \, \int f d\mu \geq u} \KL{\nu}{\mu},
\]
where, by convention, the infimum of the empty set is defined as \( \infty \). When $\Omega=[d]=\sset{1,\dots,d}$ is finite, we shall also use the notation $\Kinf\pa{\pa{p_i}_{i\in [d]},u,\pa{f_i}_{i\in [d]}}$, for $f\in \R^d$ and $p\in \cM_1\pa{[d]}$.
Like the KL divergence, $\Kinf$ admits a variational formula, which we state in Lemma~\ref{lem:Variational formula for Kinf}.
\begin{lemma}[Variational formula for $\Kinf$]
    For all $\nu\in \cM_1(\Omega)$, $f\in \cC(\Omega)$, $u\in [f_{\min},f_{\max})$, where $f_{\max}\triangleq \max_{x\in \Omega} f(x)$, $f_{\min}\triangleq \min_{x\in \Omega} f(x)$, we have
    \[\Kinf(\nu, u, f) = \max_{\lambda\in [0,1/(f_{\max}-u)]} \int{\log\pa{1-\lambda \pa{f-u}}}d\nu.\]
 \label{lem:Variational formula for Kinf}
\end{lemma}
\noindent
This formula appears in \cite{honda2015nonasymptotic,garivier2022klucbswitch}. 
$\mathcal{K}_{\inf}$ can be interpreted as a distance from the measure $\nu$ to a set of all measures $\mu$ satisfying the constraint, where the distance is measured by the KL-divergence.
The measure $\mu$ solving this optimization problem is called moment projection ($M$-projection) or reversed information projection ($rI$-projection), see \cite{csiszar2003information,bishop2006pattern,murphy2022probabilistic}.
Since the KL-divergence is not symmetric, this is different from the more common information projection ($I$-projection), 
\(
\inf_{\mu\in S} \KL{\mu}{\nu},
\)
 appearing, for example, in Sanov-type deviation bounds \cite{sanov1961probability}.
The $I$-projections have an excellent geometric interpretation because the KL can be viewed as a Bregman divergence. The $M$-projections are not Bregman divergences and lack geometric interpretation. However, they are deeply connected to the maximum likelihood estimation when the measure $\nu$ is the empirical measure of a sample \cite[Lemma 3.1]{csiszar2004information}. Additionally, within a Bayesian framework, $M$-projections naturally arise as a rate function in the context of large deviation principles (LDP) \cite{ganesh1999inverse}. $M$-projections also naturally appear in lower (and sometimes upper) bounds for multi-armed bandits \cite{Lai1985asymptotically,burnetas1996optimal}.

\section{A tighter perturbed-KL for Beta distributions}
\label{sec:beta}
In this section, we focus on the Beta distribution \(\text{Beta}(a, b),\) with parameters \( a, b > 0 \). We state in Lemma~\ref{lem:citephenzi2023some} a more general Beta bound than that of \cite{henzi2023some}. In their formulation, 
the perturbation parameter \( \eta \) does not appear explicitly;  
the result is obtained by directly setting \( \eta = 1 + \frac{a - 1}{b + 1} \) in the bound, only considering the regime \( a \geq 1 \).
As a preliminary step, we reformulate the statement by allowing \( \eta \) to range over  
\([0, \min(a,\, 1 + \frac{a - 1}{b + 1})]\),  
where taking the minimum with \(a\) removes the need to assume \(a \geq 1\).  
While choosing \(\eta\) below the upper bound is clearly suboptimal in the Beta case,  
this reformulation sets the stage for the Dirichlet setting, where such flexibility becomes relevant.

The next step, and first main result of this work,  
is to show that the perturbed KL bound remains valid  
for values of the perturbation parameter \(\eta\) beyond  
\( 1 + \frac{a - 1}{b + 1} \), with an enlarged upper bound that depends 
not only on \(a\) and \(b\), but also on \(u\).
This new upper bound is defined as
$$
S(a,b,u)\triangleq  \left\{
    \begin{array}{ll}
      a & \mbox{if } a\leq 1 \mbox{ or } u=0,\\
    a-b\frac{a-1}{b+1} & \mbox{if } a>1, u = 1,
    \\
        a - b\frac{u}{1-u} + \frac{W_0\pa{ \frac{b u^{a - b\frac{u}{1-u}} }{1 - u}\log\pa{\frac{1}{u}}}}{\log\pa{\frac{1}{u}}} & \mbox{if } a> 1, u\in (0,1), 
      
    \end{array}
\right.
$$
where $W_0$ is the principal branch of the Lambert W function. Notice, we get back the traditional constant bound
\(
\min\pa{a,1 + \frac{a - 1}{b + 1}}
\) when $u=1$. We have that \( S(a, b, \cdot) \) is non-increasing (point \((ii)\) of Proposition~\ref{prop:nonincreasing}), so the farther \( u \) is from \( 1 \), the larger the possible perturbation could be.

Since the exact expression for this upper bound involves the Lambert W function,  
which may be cumbersome to use in practice, we also provide the following  
explicit lower bounds for $a>1$ and $u\in (0,1)$:
\begin{align*} 
S(a,b,u)&\geq a-b\pa{b\log\pa{\frac{1}{u}}+\frac{1}{u}-1}\frac{\sqrt{1+\frac{2\pa{\frac{1}{u}-1}^{2}\log\pa{\frac{1}{u}}(a-1)}{\pa{b\log\pa{\frac{1}{u}}+\frac{1}{u}-1}^2}}-1}{\pa{\frac{1}{u}-1}^{2}}
\\&\geq a-b\frac{a-1}{b+\pa{\frac{1}{u}-1}/\log\pa{\frac{1}{u}}}\\&\geq a-b\frac{a-1}{b+1} = 1 + \frac{a - 1}{b + 1},
\end{align*}
where the last lower bound is the one used in \cite{henzi2023some}.

The choice of \( S(a,b,u) \) is motivated by the proof technique: to upper-bound \( h(u) = \PP{B \geq u} \) by  
\(
g(u) = \PP{B \geq x} \exp\big(-t \kl{x}{u}\big),
\)
with \( x = \frac{a - \eta}{a + b - \eta} \) and \( t = a + b - \eta \), the analysis of the derivatives \( g' \) and \( h' \) reduces to ensuring that 
\begin{align}
\label{eq:fondeq}
\frac{u - x}{1 - x} u^{-\eta} \leq 1,
\end{align}
which holds by construction whenever \( \eta \leq S(a,b,u) \) (see Proposition~\ref{prop:nonincreasing}).
The aforementioned lower bounds are derived by first taking the logarithm of both sides of \eqref{eq:fondeq}. We then use the inequality \( \log(1 - x) \leq -x - \frac{x^2}{2} \leq -x, \) for \( x \geq 0 \), 
to relax the logarithmic expression.

The proof technique for the following lemma draws inspiration from the approach used in \cite{henzi2023some}. Specifically, the proof involves examining the behavior of a function~\( J \). In \cite{henzi2023some}, when the perturbation \(\eta\) is equal to \(1 + \frac{a - 1}{b + 1}\), the authors conclude the proof by demonstrating that \( J \) is monotonic on $[(a-\eta)/(a+b-\eta), 1]$. In our analysis, by exploiting the expression of \( J \) more thoroughly, we conclude under a weaker condition on \( J \), specifically that \( J - 1 \) changes its sign at most once on $[(a-\eta)/(a+b-\eta), 1]$, allowing us to be more aggressive in determining the upper bound of \(\eta\).

\begin{lemma}
Let $B\sim\text{Beta}(a,b)$, where $a,b> 0$.
Let $u\in [0,1]$ and $\eta\in [0,S(a,b,u)]$.  Then, if $u> (a-\eta)/(a+b-\eta)$,
\[\log\PPs{}{B\geq u} - \log\PP{B\geq \frac{a-\eta}{a+b-\eta}} \leq -(a+b-\eta)\kl{\frac{a-\eta}{a+b-\eta}}{u}.\]
\label{lem:citephenzi2023some}
\end{lemma}
\begin{proof} 
Let $B\sim\text{Beta}\pa{tx + \eta, t(1-x)}$, with $t=a+b-\eta$ and $x=\frac{a-\eta}{t}$. We assume that $u>x$.
Consider $g(v)\triangleq \PP{B \geq x}\exp\pa{-t \kl{x}{v}}$ and $h(v)\triangleq \PP{B \geq v}$. We have $g(x)=h(x)$ and $g(1)=h(1)=0$. Moreover, 
\begin{align*}h'(v) &= \frac{-\Gamma\pa{t+\eta}}{\Gamma\pa{tx + \eta}\Gamma\pa{t(1-x)}}v^{tx + \eta -1}(1-v)^{t(1-x) - 1}\leq 0,\\g'(v) &= \PP{B \geq x} t\pa{\frac{v}{x}}^{-\eta}\pa{ \frac{1-v}{1-x}- \frac{v}{x} }\pa{\frac{v}{x}}^{tx + \eta - 1} \pa{\frac{1-v}{1-x}}^{t(1-x) - 1} \\&= \underbrace{\PP{B \geq x} \frac{tx^{\eta - 1}\Gamma\pa{tx + \eta}\Gamma\pa{t(1-x)}}{(1-x)\Gamma\pa{t+\eta}} {v}^{-\eta}\pa{v - {x}{}}}_{J(v)} h'(v),\\J'(v) &= \PP{B \geq x}  \frac{tx^{\eta - 1}\Gamma\pa{tx + \eta}\Gamma\pa{t(1-x)}}{(1-x)\Gamma\pa{t+\eta}}\pa{1+\eta\pa{\frac{x}{v} - 1} }v^{-\eta}.
\end{align*}
$J(x)=0$, so the proof is done if we prove that \( 1 - J \) either maintains a constant sign on \( (x, u) \) or changes sign exactly once on \( (x, u) \) but not on \( (u,1) \). In the first case, we obtain  
\[
g(u)-h(u) = \int_{x}^u \pa{1-J(v)}(-h'(v))dv \geq 0,
\]  
while in the second case, we have  
\[
g(u)-h(u) = \int_{u}^1 \pa{J(v)-1}\pa{-h'(v)}dv \geq 0.
\]
If $\eta \in [0, 1/(1-x)]$, then we see that $J'\geq 0$, so \( J - 1 \) changes its sign at most once on \( (x,1) \). If $\eta>1/(1-x)$, \( J \) increases from \( 0 \) to a maximum over the segment $[x,\frac{x}{1-1/\eta}]$ and then decreases on $[\frac{x}{1-1/\eta}, 1]$. If \( J(1) > 1 \), then \( J - 1 \) changes its sign exactly once on \( (x,1) \). Now, if $J(1)\leq 1$, we have $J(u)=J(1)\frac{u-x}{1-x}u^{-\eta}\leq \frac{u-x}{1-x}u^{-\eta},$ which is upper bounded by~$1$ from point $(i)$ of Proposition~\ref{prop:nonincreasing} (it is at this point that the definition of $S$ is invoked). We thus need to determine if \( u \) is lower (no change of sign on \([x, u]\)) or higher (two potential changes of sign on \([x, u]\)) than \(\frac{x}{1 - 1/\eta}\).

Isolating $x$, we get $\frac{u-u^{\eta}}{1-u^{\eta}} \leq x$. Thus, from the inequality $u^{1-\eta}+(\eta-1)u-\eta \geq 0$ (obtained by noting that the derivative with respect to \( u \) is negative), we get 
\(
 (1-1/\eta)u \leq \frac{u-u^{\eta}}{1-u^{\eta}} \leq x ,
\)
so $u\in [x,\frac{x}{1-1/\eta}]$, i.e., $1-J\geq 0$  on $[x,u]$.
\end{proof}

The following corollary expresses the above tail bound for a Beta random variable~\(B\) in a form involving the Dirichlet vector \((B, 1 - B)\). This formulation is useful as it naturally generalizes to the multivariate context of Dirichlet distributions and DPs. Moreover, the upper bound on \(\eta\) is designed to not depend directly on the weighting in the vector \((B, 1 - B)\), but rather on some lower bounds on these weights (leveraging the non-increasing monotonicity of \( S(a, b, \cdot) \)), which is interesting when they are themselves independent random variables that can't be used outside the probability, e.g., when they form Dirichlet-weighted sums. 

\begin{corollary}
\label{cor:citephenzi2023some}
Let $1\geq v \geq v'>u\geq 0$ and $1\geq w\geq w'\geq 0$. Let $B\sim\text{Beta}(a,b)$, where $a,b> 0$. Let  \(0\leq\eta\leq S\pa{a,b,\frac{(u-w')^+}{v'-u+(u-w')^+}}.\) 
Then, \[
\log\PP{Bv + (1-B)w \geq u}
\leq 
-(a+b-\eta)\Kinf\pa{\begin{pmatrix} \frac{a-\eta}{a+b-\eta} \\ \frac{b}{a+b-\eta} \end{pmatrix} , u, \begin{pmatrix}   v \\ w \end{pmatrix}}.
\]
\end{corollary}
\begin{proof}
 $S(a,b,\cdot)$ is non-increasing from point $(ii)$ of Proposition~\ref{prop:nonincreasing}. Therefore, we have $S\pa{a,b,\frac{(u-w)^+}{v-u+(u-w)^+}}\geq S\pa{a,b,\frac{(u-w')^+}{v'-u+(u-w')^+}}$, and from Lemma~\ref{lem:citephenzi2023some},
\begin{align*}\log\PP{Bv + (1-B)w \geq u} &= \log\PP{B \geq {\frac{(u-w)^+}{v-u+(u-w)^+}}} \\&\leq -(a+b-\eta)\kl{\frac{a-\eta}{a+b-\eta}}{{ \frac{(u-w)^+}{v-u+(u-w)^+}}} \\&=-(a+b-\eta)\Kinf\pa{\begin{pmatrix} \frac{a-\eta}{a+b-\eta} \\ \frac{b}{a+b-\eta} \end{pmatrix} , u, \begin{pmatrix}   v \\ w \end{pmatrix}}.\end{align*}    
\end{proof}

\section{Perturbed-KL for Dirichlet distributions and DPs}
In this section, we focus on Dirichlet distributions and, more generally, on DPs. Specifically, we consider a base distribution \(\nu_0 \in \mathcal{M}_1(\Omega)\) with a scale parameter \(\alpha > 0\). The case of the Dirichlet distribution is obtained by considering a finite space \(\Omega\).
Building upon Corollary~\ref{cor:citephenzi2023some}, we extend the perturbed-KL bounds to \(X\sim\DP(\alpha \nu_0)\) in Theorem~\ref{thm:main}.
More precisely, we consider the deviation of a Dirichlet-weighted sum (or integral for DPs), and the binary KL divergence is replaced by \(\Kinf\) in the upper bound (in the same way as in Corollary~\ref{cor:citephenzi2023some}). It is important to note that the perturbation is now represented by some measure \(\eta \in \mathcal{M}(\Omega)\). For the sake of consistency, we retain the notation \(\eta\), keeping in mind that in this context, it is not a scalar. 

The generalization relies on introducing a measure $\eta_0 \in \mathcal{M}(\Omega)$, with $\eta_0 \leq \alpha \nu_0$, which specifies the region where the perturbation may act (via $\eta \leq \eta_0$).  
More precisely, the perturbation can only target the lowest values of $f$ within $\eta_0$.
 We note that in $\Kinf$, the perturbation has the strongest effect on the largest values of $f$.  
Therefore, $\eta_0$ should ideally be concentrated on regions where $f$ attains high values.
The function~$S$ still determines the maximal admissible perturbation, now controlling the mass $m$ of $\eta$.  
By Proposition~\ref{prop:nonincreasing}, $(ii)$, $S$ is non-decreasing in $\eta_0(\Omega)$.   
When $\eta_0(\Omega) \leq 1$, we may take $m = S= \eta_0(\Omega)$, yielding $\eta = \eta_0$ which can be concentrated on the largest values of~$f$.   
For $\eta_0(\Omega) > 1$, we have $m \leq S< \eta_0(\Omega)$, so some mass is lost when passing from $\eta_0$ to~$\eta$; by construction, this loss occurs on the highest values of $f$.  
Thus, increasing $\eta_0(\Omega)$ provides more available mass but may reduce the perturbation at $f$'s maximum, creating a trade-off.
For instance, if $\Omega$ is finite and $\alpha\nu_0$ assigns more than one unit of mass to the maximizers of $f$, one may take $\eta_0$ concentrated on these points, equal to $\alpha\nu_0$, as both the loss and the gain occur at $f$'s maximizers.   
In contrast, if $\alpha\nu_0$ assigns less than one unit to the maximizers and the second-largest value of $f$ is much smaller, reallocating mass away from the maximum to increase $\eta_0(\Omega)$ may be detrimental.

The proof of Theorem~\ref{thm:main} builds upon the decomposition $\Omega = I \sqcup J$, where $I = \operatorname{supp}(\eta_0)$ is the region on which the perturbation can directly act, and $J$ is its complement. Then, we consider the beta sample defined by 
$B=X(I)$, where $X\sim\DP\pa{\alpha \nu_0}$, 
and see that $\int{f}dX=B \int_I f dX / X(I) + (1-B)\int_J f dX / X(J)$. The trick is then to see that $B$ and $\int_I f dX/X(I)$ (resp. $\int_J f dX/X(J)$) are independent (so that Corollary~\ref{cor:citephenzi2023some} can be used). Then, what we obtain is an expectation (over some DP) of some exponential bound, with a binary KL as a rate function. Leveraging the connection between \( \Kinf \) and the MGF of the Gamma process, we can convert this into the desired  \( \Kinf \) exponential bound. 

One remaining obstacle in the proof is that the weights of the Beta vector are still perturbed.  
In the present setting, since $\eta$ is by definition concentrated on the lower values of $f$ within $I$, we can replace the perturbed DP with an unperturbed DP in the weights.
\begin{theorem}
Let $f\in\cC\pa{\Omega}$ and $u\in [f_{\min},f_{\max})$, where $f_{\max}\triangleq \max_{\Omega} f$, $f_{\min}\triangleq \min_{\Omega} f$. 
Let $\eta_0\in \cM\pa{f^{-1}\pa{(u,\infty)}}$ such that $\eta_0\leq \alpha \nu_0$. Let $m \geq 0$ such that \[m \leq S\pa{\eta_0(\Omega),\alpha-\eta_0(\Omega),\frac{(u-f_{\min})^+}{\min_{\text{supp}(\eta_0)}f-u+(u-f_{\min})^+}}.\] 
We consider \(
\eta \triangleq \argmin_{\mu\leq \eta_0,~\mu(\Omega)=m} \int f d\mu.
\)
Then,
\[\log \PPs{X\sim\DP(\alpha \nu_0)}{\int f dX\geq u}\leq -{\pa{{\alpha - m}}\Kinf\pa{\frac{\alpha\nu_0-\eta}{\alpha - m}, u, f}}.\] 
\label{thm:main}
\end{theorem}
\begin{proof}
Up to replacing \(\Omega\) by the product metric space \(\Omega \times \{0,1\}\) equipped with
\(
\mathcal{B}(\Omega \times \{0,1\}) = \mathcal{B}(\Omega) \otimes \mathcal{P}(\{0,1\}),
\)
and \(\alpha \nu_0\) by the measure defined as
\(
A \times \{0\} \mapsto \eta_0(A),\) and \(A \times \{1\} \mapsto \alpha \nu_0(A) - \eta_0(A),
\) for $A\in \mathcal{B}(\Omega),$
we may assume, without loss of generality, that \(\alpha \nu_0\) coincides with \(\eta_0\) on
\(
I \triangleq \operatorname{supp}(\eta_0).\) Let $J\triangleq \Omega \backslash I$, $t\triangleq \alpha - \eta\pa{\Omega}$ and $\nu\triangleq (\alpha \nu_0 - \eta)/t$. 

We use the convention $\min_{\emptyset} f=\max_{\emptyset} f=\frac{\int_A f d\mu}{\mu(A)}= f_{\max}$, for all $A\in \cB(\Omega)$ and $\mu\in \cM(\Omega)$ such that $\mu(A)=0$.
From the definition of $\eta$, for all $x\in \text{supp}(\eta_0-\eta)$ and all $y\in \text{supp}(\eta)$, we have $x\geq y$.
We consider $G, G', B\sim \cG\pa{t\nu} \otimes \cG\pa{\eta} \otimes\text{Beta}\pa{(t\nu+\eta)(I),t\nu(J)}$. Then, \[B \frac{G(\cdot\cap I)+G'}{(G+G')(I)} +(1-B)\frac{G(\cdot\cap J)}{G(J)} \sim \DP(t\nu+\eta).\]
Since
\[\frac{\int_I f dG'}{G'(I)} \leq \max_{\text{supp}(\eta)}f \leq \min_{\text{supp}(\eta_0-\eta)}f \leq\frac{\int_I f dG}{G(I)}\quad \text{a.s.,}\]
we have that 
\[\frac{\int_I f d(G+G')}{(G+G')(I)} \leq \frac{\int_I f dG}{G(I)}\quad \text{a.s.,}\]
thus giving
\begin{align}
    &\nonumber\PP{B \frac{\int_I f d(G+G')}{(G+G')(I)} +(1-B)\frac{\int_J f dG}{G(J)} \geq u}
    \leq\PP{B \frac{\int_I f dG}{G(I)} +(1-B)\frac{\int_J f dG}{G(J)} \geq u}\\
    &= \EEs{X\sim \DP(t\nu)}{\PPcc{B \frac{\int_I f dX}{X(I)} + (1-B) \frac{\int_J f dX}{X(J)}}{X}} \nonumber \\
    &\leq
    \EEs{X\sim \DP(t\nu)}{\exp\pa{-t\Kinf\pa{\begin{pmatrix}
        \nu(I) \\ \nu(J)
    \end{pmatrix}
  , u, \begin{pmatrix}
        \frac{\int_I f dX}{X(I)} \\ \frac{\int_J f dX}{X(J)}
    \end{pmatrix}}{}}}\label{citephenzi2023some}
    \\&\leq\label{sup_PI}
    \exp\pa{-t\Kinf\pa{\nu,u,f}}.
    \end{align}
    where \eqref{citephenzi2023some} is from Corollary~\ref{cor:citephenzi2023some} and \eqref{sup_PI}
is from Proposition~\ref{prop:sup_PI}. 
\end{proof}
\section{Conclusion}
In this work, we began by refining the existing perturbed KL bounds for the Beta distribution, allowing for larger perturbations and thus providing tighter bounds. 
We then generalized these improved bounds to Dirichlet distributions and DPs, where the perturbation in this context is represented by a measure, opening up new avenues for research and application in statistical theory and machine learning.

\section{Appendix}
\begin{proposition}
\label{prop:nonincreasing}
 Let $a > 1$, $b>0$. For $u\in (0,1)$ and $\eta\geq 0$, we define
\begin{align*} 
R(\eta, u)&\triangleq\frac{u-\frac{a-\eta}{a+b-\eta}}{1-\frac{a-\eta}{a+b-\eta}}u^{-\eta} -1 = \frac{ub-{(1-u)(a-\eta)}}{b}u^{-\eta} -1.
\end{align*}
Then,
\begin{itemize}
    \item[$(i)$] For a fixed $u\in (0,1)$, the function $R(\cdot,u)$ admits a unique root in $\R_+$, equal to $S(a,b,u)$. We have $S(a,b,u)<a$ and $R(\eta, u)\leq 0$ for $\eta\in [0,S(a,b,u)]$.
    \item[$(ii)$] $S(a,b,\cdot)$ is non-increasing on $(0,1)$. 
    $S(\cdot,a+b-\cdot,u)$ is non-decreasing on $(0,a+b)$.
\end{itemize}
\end{proposition}
\begin{proof}
We provide the proof for each point individually.
\begin{itemize}
\item[$(i)$] Let $u\in (0,1)$. The equation $R(\eta, u)=0$ in the variable $\eta$ is equivalent to
\begin{align*}
    \pa{b\frac{u}{1-u}+\eta -a}\log\pa{\frac{1}{u}}\exp\pa{\pa{\eta-a+b\frac{u}{1-u}}\log\pa{\frac{1}{u}}} = \frac{b{u}^{a-b\frac{u}{1-u}}}{1-u}\log\pa{\frac{1}{u}}.
\end{align*}
Thus, we find that $\pa{b\frac{u}{1-u}+\eta -a}\log\pa{\frac{1}{u}}$ is the unique solution to a transcendental equation involving the Lambert W function. Since the RHS is positive, we have \[\pa{b\frac{u}{1-u}+\eta -a}\log\pa{\frac{1}{u}} = W_0\pa{ \frac{b u^{a - b\frac{u}{1-u}} }{1 - u}\log\pa{\frac{1}{u}}},\]
leading to $\eta=S(a,b,u)$. From $a>1$, we have
\begin{align*}
u^a<u &\iff 
\frac{b u^{a - b\frac{u}{1-u}} }{1 - u}\log\pa{\frac{1}{u}} < \log\pa{\frac{1}{u}} b\frac{u}{1-u}\exp\pa{\log\pa{\frac{1}{u}} b\frac{u}{1-u}},
\\ &\iff {W_0\pa{ \frac{b u^{a - b\frac{u}{1-u}} }{1 - u}\log\pa{\frac{1}{u}}}}{} < \log\pa{\frac{1}{u}} b\frac{u}{1-u},\\
    &\iff  S(a,b,u)<a.
\end{align*}
Since \( R(0, u) = \frac{b u - a(1 - u)}{b} - 1 < 0 \) and the root $S(a,b,u)$ of $R(\cdot,u)$ is unique, by the continuity of this function, it follows that \( R(\eta, u) \leq 0 \) for \( \eta\in [0, S(a,b,u)] \).
\item[$(ii)$] Since $$S(a,b,u)=\sup{\sset{\eta\in[0,a],~{{\frac{u-\frac{a-\eta}{a+b-\eta}}{1-\frac{a-\eta}{a+b-\eta}}u^{-\eta}}}<1}},$$
we first see that $S(\cdot,a+b-\cdot,u)$ is non-decreasing on $(0,a+b)$, as the feasible sets for $\eta$ are nested when the variable increases. The same argument also proves that  $S(a,b,\cdot)$ is non-increasing. More precisely, for $\eta\in [0,a]$ and $u\in\pa{0,1}$, let $f(\eta,u)\triangleq \log(\max\pa{0,1+R(\eta, u)}).$ Then, it is sufficient to show that $f(\eta,\cdot)$ is non-decreasing on $\pa{\frac{{a-\eta}}{{a+b-\eta}},1}$. We have
\[\frac{\partial f(\eta,u)}{\partial u} = \frac{\log\pa{1-\frac{\frac{1-u}{u}(a-\eta)}{b}} + \frac{\pa{a-\eta}\log(1/u)}{\pa{a+b-\eta}{}u - \pa{a-\eta}}}{u\log^2(1/u)},\]
which is non-negative because the numerator is a sum of two non-increasing functions, both of which tend to \( 0 \) as \( u \to 1^- \). 
\end{itemize}
\end{proof}

\begin{proposition}\label{prop:sup_PI}
Let $t>0$, $\nu\in \cM_1(\Omega)$, $f\in\cC\pa{\Omega}$ and $u\in [f_{\min},f_{\max})$, where $f_{\max}\triangleq \max_{\Omega} f$, $f_{\min}\triangleq \min_{\Omega} f$. 
Let $\Pi$ be the set of finite measurable partitions of $\Omega$.
Then
\begin{align*}&\exp\pa{-t\Kinf\pa{\nu,u,f}} \\&= \sup_{(A_1,\dots, A_k)\in \Pi} \EEs{X\sim \DP\pa{t\nu}}{\exp\pa{-t \Kinf\pa{\pa{\nu(A_i)}_{i\in [k]},u,\pa{\frac{\int_{A_i} f dX}{X(A_i)}}_{i\in [k]}}}}.\end{align*}
\end{proposition}
\begin{proof}
Let $A_1\sqcup~\dots~\sqcup A_k=\Omega$ be a finite measurable partition of $\Omega$. 
For $p\in \cM_1(\Omega)$, let $\ell(p)\triangleq \Kinf\pa{\pa{\nu(A_i)}_{i\in [k]},u,\pa{{\int_{A_i} f }/{p(A_i)}dp}_{i\in [k]}}$. Then
    \begin{align*}
\ell(p)&={\max_{\lambda \in [0,{1}/{(\max_{i\in [k]}{{{\int_{A_i} f }/{p(A_i)}dp}}-u)}]} \sum_{i\in [k]}\nu(A_i)\log\pa{1-\lambda\pa{{\frac{\int_{A_i} fdp }{p(A_i)}}-u}}}
    \\&=
    \max_{\lambda \in [0,{1}/{(\max_{i\in [k]}{{{\int_{A_i} f }/{p(A_i)}dp}}-u)}]} -t^{-1}\log\EEs{G\sim \cG\pa{t\nu}}{e^{\lambda\sum_{i\in [k]}G(A_i)({{\int_{A_i} f }/{p(A_i)}dp}-u)}}
    \\&\geq \max_{\lambda \in [0,1/\pa{f_{\max}-u}]} -t^{-1}\log\EEs{G\sim \cG\pa{t\nu}}{e^{\lambda\sum_{i\in [k]}G(A_i)({{\int_{A_i} f }/{p(A_i)}dp}-u)}},
    \end{align*}
so, $\EEs{X\sim \DP\pa{t\nu}}{\exp\pa{-t \ell(X)}}$ is upper bounded by
    \begin{align*}
&\EEs{X\sim \DP\pa{t\nu}}{\min_{\lambda \in [0,1/\pa{f_{\max}-u}]} \EEccs{G\sim \cG\pa{t\nu}}{e^{\lambda\sum_{i\in [k]}G(A_i)({{\int_{A_i} f }/{X(A_i)}dX}-u)}}{X}}
    \\&\leq
    \min_{\lambda \in [0,1/\pa{f_{\max}-u}]}\EEs{X\sim \DP\pa{t\nu}}{ \EEccs{G\sim \cG\pa{t\nu}}{e^{\lambda\sum_{i\in [k]}G(A_i)({{\int_{A_i} f }/{X(A_i)}dX}-u)}}{X}}
\\&=
    \min_{\lambda \in [0,1/\pa{f_{\max}-u}]}{ \EEs{G\sim \cG\pa{t\nu}}{e^{\lambda\pa{\int (f-u)dG}}}} = \exp\pa{-t\Kinf\pa{\nu,u,f}}.
\end{align*}
This is true for any finite partition, so $\exp\pa{-t\Kinf\pa{\nu,u,f}}$ is lower bounded by
\begin{align*}\sup_{(A_1,\dots, A_k)\in \Pi} \EEs{X\sim \DP\pa{t\nu}}{\exp\pa{-t \Kinf\pa{\pa{\nu(A_i)}_{i\in [k]},u,\pa{\frac{\int_{A_i} f dX}{X(A_i)}}_{i\in [k]}}}}.\end{align*}
To have the equality, let $\varepsilon>0$. There exists $(A_1,\dots, A_k)\in \Pi$ such that for any $p\in \cM_1(\Omega)$, $\left\| \sum_{i \in [k]} \frac{\int_{A_i} f\, dp}{p(A_i)} \II{\cdot \in A_i} - f \right\|_{\infty} \leq \varepsilon
$. Thus, 
\begin{align*}\exp\pa{-t\Kinf\pa{\nu,u - \varepsilon,f}} &= \exp\pa{-t\Kinf\pa{\nu,u,f+\varepsilon}}\\&\leq \EEs{X\sim \DP\pa{t\nu}}{\exp\pa{-t \Kinf\pa{\pa{\nu(A_i)}_{i\in [k]},u,\pa{\frac{\int_{A_i} f dX}{X(A_i)}}_{i\in [k]}}}}.\end{align*}
We get our result using the left-continuity of $\Kinf\pa{\nu, \cdot,f}$ \cite{garivier2022klucbswitch}.
\end{proof}
